\documentclass[12pt]{article}
\usepackage[utf8]{inputenc}

\usepackage{amsfonts}
\usepackage{amssymb}
\usepackage{amsmath}
\usepackage{amsthm}
\usepackage{amscd}
\usepackage{mathrsfs}
\usepackage{pbox}
\usepackage[T2A]{fontenc}

\usepackage[english]{babel}

\usepackage{makecell}%To keep spacing of text in tables
\setcellgapes{4pt}%parameter for the spacing

\usepackage{amsthm}

\def \le {\leqslant}
\def \ge {\geqslant}

\newcommand{\centered}[1]{\begin{tabular}{l} #1 \end{tabular}}

\theoremstyle{plain}
\newtheorem{theorem}{Theorem}

\newtheorem{lem}{Lemma}
\newtheorem{foll}{Corollary}

\addto\captionsrussian{}
\addto\captionsrussian{}
\newcommand\blfootnote[1]{%
  \begingroup
  \renewcommand\thefootnote{}\footnote{#1}%
  \addtocounter{footnote}{-1}%
  \endgroup
}

\begin{document}

\begin{center}
\textsc{\Large Discrete part of the second Lagrange spectrum}\\ 

\,
\,
\large Dmitry Gayfulin\footnote{This work was supported by the Russian Science Foundation under grant No. 22-41-05001.}
\end{center}

\begin{abstract}
\noindent
Given an irrational number $\alpha$ consider its irrationality measure function $\psi_{\alpha}(t)=\min\limits_{1\le q\le t, q\in\mathbb{Z}}\|q\alpha\|$. The set of all values of $\lambda(\alpha)=(\limsup\limits_{t\to\infty} t\psi_{\alpha}(t))^{-1}$ where $\alpha $ runs through the set $\mathbb{R}\setminus\mathbb{Q}$ is called the Lagrange spectrum $\mathbb{L}$. In the paper \cite{Moshchevitin} an irrationality measure function $\psi^{[2]}_{\alpha}(t)=\min\limits_{1\le q\le t, q\in\mathbb{Z},q\ne q_i}\|q\alpha\|$  was introduced. In other words, we consider the best approximations by fractions, whose denominators are not the denominators of the convergents to $\alpha$. Replacing the function $\psi_{\alpha}$ in the definition of $\mathbb{L}$ by $\psi^{[2]}_{\alpha}$, one can get a set $\mathbb{L}_2$ which is called the ''second'' Lagrange spectrum. In this paper we give the complete structure of discrete part of $\mathbb{L}_2$.
\end{abstract}

\vskip+1cm
\section{Introduction}

\vskip+0.3cm
Consider an arbitrary irrational number $\alpha$. Denote by $a_i,\ i\in\mathbb{Z}_{\ge0}$ the partial quotients of its continued fraction expansion 
\begin{equation}\label{q}
\alpha = [a_0; a_1, a_2, \ldots] = a_0 + \cfrac{1}{a_1+\cfrac{1}{a_2+\ldots}}, ~a_0\in\mathbb{Z}, ~a_i\in\mathbb{Z}_{+}, ~i = 1, 2, \ldots.
\end{equation}
Let us denote by $\frac{p_n}{q_n}$ the $n$-th convergent fraction to $\alpha$, i.e. $\frac{p_n}{q_n}:=[a_0;a_1,\ldots,a_n]$. Define $\mathcal{Q}_{\alpha}:=\{q_1,q_2,\ldots\}$ to be the sequence of denominators of the convergents to $\alpha$.
The irrationality measure function of $\alpha$ is defined as
$$
\psi_{\alpha}(t):=\min\limits_{1\le q \le t, q\in\mathbb{Z}} \|q\alpha\| ,\,\,\,\,\,\text{where}\,\,\,\,\,  \|x\| = \min\limits_{n\in\mathbb{Z}}|x-n|.
$$ 
By Lagrange's theorem on best approximations (see \cite{khinchin1964}), $\psi_{\alpha}(t)$ is a piecewise constant function, moreover
$$ \psi_{\alpha}(t) = |q_{n}\alpha - p_{n}|=\|q_n\alpha\|
\,\,\,\,\,\text{
for}
\,\,\,\,\,
q_n\le t < q_{n+1}.
$$
Let
$$
\alpha_n:=[a_n;a_{n+1},\ldots],\quad \alpha_n^{*}:=[0;a_n,a_{n-1},\ldots,a_1].
$$
We say that two irrational numbers $\alpha$ and $\beta$ are \textit{equivalent} if $\alpha_n=\beta_m$ for some $n, m\in\mathbb{N}$. It is well known fact, usually called Perron's formula, that for any $n\in\mathbb{N}$ one has
$$
\|q_n\alpha\|=\frac{1}{q_n(\alpha_{n+1}+\alpha_n^*)}.
$$
See \cite{Cusick}, Appendix 1 for details. The Lagrange constant of an irrational number $\alpha$ is defined as
 \begin{equation}
 \label{lagconstdef}
 \lambda(\alpha):=\limsup_{t\to \infty} (t\cdot \psi_\alpha (t))^{-1}=\limsup_{t\to \infty}(\alpha_{t+1}+\alpha_t^*).
 \end{equation}
 The set of values $\lambda(\alpha)$ as $\alpha$ runs through the set of all irrational numbers forms the
 Lagrange spectrum $\mathbb{L}$
\begin{equation}
\label{lagrdef}
\mathbb{L}:=\{\lambda ~|~\exists\alpha\in\mathbb{R\setminus Q} \colon \lambda = \lambda(\alpha)  \}.
\end{equation}
One can find a nice and detailed survey of the results about the Lagrange spectrum up to 1990s in a book \cite{Cusick}.
\\
\textbf{Remark.} Some authors, for example \cite{Moshchevitin}, define the Lagrange constant $\lambda(\alpha)$ equal to $\liminf\limits_{t\to \infty} t\cdot \psi_\alpha (t)$. Therefore, the Lagrange spectrum, defined by these authors, contains the reciprocals of the elements of $\mathbb{L}$ from our definition. 

In \cite{Moshchevitin} Moshchevitin considered an irrationality measure function
\begin{equation}
\psi^{[2]}_{\alpha}(t):=\min\limits_{1\le q\le t, q\in\mathbb{Z}, q\notin \mathcal{Q}_{\alpha}}\|q\alpha\|.
\end{equation}
related to so-called second best approximations. Obviously, $\psi^{[2]}_{\alpha}(t)$ is also a non-increasing peacewise constant function. It is connected with the properties of best approximations by fractions, whose denominators are not the denominators of the convergents to $\alpha$. Also, in \cite{Moshchevitin} an analogue of Lagrange constant was defined
\begin{equation}
\label{lagconstdef2}
\lambda^{[2]}(\alpha):=\limsup_{t\to \infty} (t\cdot \psi^{[2]}_\alpha (t))^{-1}
\end{equation}
and the corresponding spectrum 
\begin{equation}
\label{lagrdef}
\mathbb{L}_2:=\{\lambda ~|~\exists\alpha\in\mathbb{R\setminus Q} \colon \lambda = \lambda^{[2]}(\alpha)  \}
\end{equation}
was considered. The structure of $\mathbb{L}_2$ in some sense is similar to the structure of $\mathbb{L}$. The smallest elements of $\mathbb{L}_2$ form a countable discrete set and $\mathbb{L}_2$ contains a Hall's ray i.e. $\exists \mu: (\mu,+\infty)\in\mathbb{L}_2$. 

However, in the present paper we show that, unlike the classical Lagrange spectrum $\mathbb{L}$, whose discrete part is connected with the properties of so-called Markoff triples, the discrete part of $\mathbb{L}_2$ has a short and simple description. We prove that if a real number $x$ belongs to the discrete part of $\mathbb{L}_2$, then $x$ is equivalent to one of the following numbers: $[1;\overline{1}]$, $[1;\overline{1,1,3}]$ or $[1;\overline{1,1,1,1,3,(1,1,3)_{2k-1}}]$, where $k\ge 1$. We write $(a_1,\ldots,a_n)_m$ if the pattern $a_1,\ldots,a_n$ is repeated $m$ times and $\overline{a_1,\ldots,a_n}$ if this pattern is repeated infinitely many times.

It turns out that the discrete part of $\mathbb{L}_2$ is similar to the discrete part of so-called Dirichlet spectrum $\mathbb{D}$, which is defined as follows. Denote
$$
d_n(\alpha):=[a_n;a_{n-1},\ldots,a_1][a_{n+1};a_{n+2},\ldots]
$$
and
$$
d(\alpha):=\limsup\limits_{n\to\infty} d_n(\alpha).
$$
Then $\mathbb{D}$ is defined as a set of values taken by $d(\alpha)$ as $\alpha$ runs through the set of all irrational numbers. Spectrum $\mathbb{D}$ was studied by several authors (see for example \cite{Divis}, \cite{Lesca}, \cite{Morimoto}). It was shown by Lesca \cite{Lesca} that if $d(\alpha)<2+\sqrt{5}$, which is the accumulation point of the discrete part of $\mathbb{D}$, then $\alpha$ is equivalent either to $[1;\overline{1}]$ or to $[1;\overline{(1)_{2k-1},2}]$ for some $k\ge 1$. Our proof is similar to his argument.

\section{Results by Moshchevitin and Semenyuk}
We start with the lemma from \cite{Moshchevitin} which is the main tool to calculate $\lambda^{[2]}(\alpha)$  of given irrational number $\alpha$. We follow the notation of this paper but, as we already mentioned, we consider the ''reverse'' definition of $\mathbb{L}$ and $\mathbb{L}_2$. That is why, $\varkappa^i_n(\alpha)$ in our notation equals $1/\varkappa^i_n(\alpha)$ in the notation of the paper \cite{Moshchevitin}. 
\begin{lem}
\label{kappalem}
Suppose that an irrational number $\alpha$ is not equivalent to $\frac{1+\sqrt{5}}{2}=[1;\overline{1}]$. Consider three quantities:
\begin{equation}
\label{kappadef}
\varkappa^1_n(\alpha) = \frac{\alpha_n + \alpha^{*}_{n-1}}{(1 + \alpha^{*}_{n-1})(\alpha_n - 1)},\,\,\,\,\,\,
\varkappa^2_n(\alpha) = \frac{\alpha_{n+1} + \alpha^{*}_n}{(1 - \alpha^{*}_n)(\alpha_{n+1} + 1)}
,
\,\,\,\,\,\,
\varkappa^4_n(\alpha) = \frac{\alpha_n + \alpha^{*}_{n-1}}{4}.
\end{equation}
Then
\begin{equation}
\label{lambda2formula}
\lambda^{[2]}(\alpha)=\limsup\limits_{n\to\infty\colon a_n\ge 2}\max(\varkappa^1_n(\alpha), \varkappa^2_n(\alpha), \varkappa^4_n(\alpha)).
\end{equation}
\end{lem}
In the paper \cite{Moshchevitin} the two smallest elements of $\mathbb{L}_2$ were calculated.
\\
{\bf Theorem A.}
\textit{
\begin{enumerate}
\item{The smallest element of $\mathbb{L}_2$ is $\lambda_1:=\frac{\sqrt{5}}{4}\approx 0.559016$. Moreover, if $\lambda^{[2]}(\alpha)=\lambda_1$, then $\alpha\sim\frac{1+\sqrt{5}}{2}=[1;\overline{1}]$.}
\item{The second smallest element of $\mathbb{L}_2$ is $\lambda_2:=\frac{\sqrt{17}}{4}\approx 1.030776$. Moreover, if $\lambda^{[2]}(\alpha)=\lambda_2$, then $\alpha\sim\frac{1+\sqrt{17}}{2}=[2;\overline{1,1,3}]$.}
\end{enumerate}
}
In \cite{Semenyuk} P. Semenyuk calculated the third smallest element of $\mathbb{L}_2$.
\\
{\bf Theorem B.}
\textit{
The third smallest element of $\mathbb{L}_2$ is $\lambda_3:=\frac{13\sqrt{173}}{164}\approx 1.042611$. Moreover, if $\lambda^{[2]}(\alpha)=\lambda_3$, then $\alpha\sim\frac{39+13\sqrt{17}}{82}=[2;\overline{1,1,1,1,3,1,1,3}]$.
}

\section{Main result}
For each $n\ge 3$ define 
$\xi_n=[0;\overline{1,1,1,1,3,(1,1,3)_{2n-5}}]$ and 
\begin{equation}
\label{lambdandef}
\lambda_n:=\frac{[3;\overline{(1,1,3)_{2n-5},1,1,1,1,3}]+[0;\overline{1,1,1,1,(3,1,1)_{2n-5}}]}{4}.
\end{equation}
One can easily see that $\lambda_3$ defined in (\ref{lambdandef}) is the same constant as in Theorem B. Let us also denote by $\lambda_{\infty}$ the limit
$$
\lambda_{\infty}:=\lim\limits_{n\to\infty}\lambda_n=\frac{[3;\overline{1,1,3}]+[0;1,1,1,1,\overline{3,1,1}]}{4}=\frac{3\sqrt{17}+21}{32}\approx  1.042791.
$$
Now we are ready to formulate our main result.
\begin{theorem}
\label{thm1}
\ \ 
\begin{enumerate}
\item{
The spectrum $\mathbb{L}_2$ below $\lambda_{\infty}$ forms a discrete set
$$
(-\infty,\lambda_{\infty})\cap\mathbb{L}_2=\{\lambda_1<\lambda_2<\ldots<\lambda_n<\ldots\},
$$
where $\lambda_1$ and $\lambda_2$ are defined in Theorem A, $\lambda_n, n\ge3$ is defined in (\ref{lambdandef}). Moreover, $\lambda_{\infty}\in\mathbb{L}_2$.
}
\item{For all $n\ge 3$ if an irrational $\alpha$ is such that $\lambda^{[2]}(\alpha)=\lambda_n$, then $\alpha\sim\xi_n$.}
\end{enumerate}
\end{theorem}

\section{Prohibited patterns}
The following lemma was proved in \cite{Moshchevitin} and \cite{Semenyuk}. For the sake of completeness, we give a sketch of proof of it.
\begin{lem}
\label{proh1}
Assume that the continued fraction expansion $[a_0;a_1,\ldots,a_n,\ldots]$ of an irrational number $\alpha$ contains infinitely many elements greater than $3$ or infinitely many patterns of the form $2$, $33$, $313$, $31113$, $311111$, $1113111$. Then $\lambda^{[2]}(\alpha)>\lambda_{\infty}$.
\end{lem}
\begin{proof}
See the table below. The bold script indicates the element $a_n$. To give an example, let us prove the estimate on the third line. Suppose that $a_n=2$ for infinitely many $n$. It is clear that 
$$
\lambda^{[2]}(\alpha)=\limsup\limits_{n\to\infty\colon a_n\ge 2}\max(\varkappa^1_n(\alpha), \varkappa^2_n(\alpha), \varkappa^4_n(\alpha))\ge \limsup\limits_{n\to\infty\colon a_n\ge 2} \varkappa^2_n(\alpha).
$$
Note that the function $\varkappa^1_n(\alpha_n,\alpha^*_{n-1})$ is decreasing on both arguments while the functions $\varkappa^2_n(\alpha_{n+1}, \alpha^{*}_n)$ and $\varkappa^4_n(\alpha_n,\alpha^*_{n-1})$ are increasing on both arguments. Therefore, in order to obtain the lower estimate of $\varkappa^2_n$, we need to substitute the lower estimates of $\alpha_{n+1}$ and $\alpha^{*}_n$. As we see from the lines 1 and 2 of the table below, without loss of generality one can say that $a_i\le 3$ for all $i\ge 1$. Thus, $\alpha_{n+1}\ge[1;\overline{3,1}]$ and $\alpha^*_n\ge[0;2,\overline{1,3}]$. Substituting these estimates to (\ref{kappadef}) yields  $\varkappa^2_n(\alpha)\ge 1.123722>\lambda_{\infty}$.
\\

\renewcommand{\arraystretch}{1.2}
\begin{tabular}{ | c | c | c | c | }
%\begin{tabular}{|@{}c@{}|@{}c@{}|@{}c@{}|@{}c@{}|}
\hline
Pattern & $\varkappa^i_n$ used & Lower estimate & Numerical value\\ \hline
$a_n\ge 5$  & $\varkappa^4_n$ & \pbox{20cm}{([5]+[0])/4} & $1.25$ \\ \hline
$\boldsymbol{4}$ & $\varkappa^4_n$ & $[4;\overline{4,1}]+[0;\overline{4,1}]$ & $1.103553$ \\ \hline
$\boldsymbol{2}$  & $\varkappa^2_n$ & \centered{\pbox{20cm}{$\alpha_{n+1}\ge[1;\overline{3,1}];\ \ \alpha^*_n\ge[0;2,\overline{1,3}]$}} & 1.116515 \\  [2ex] \hline
$3\boldsymbol{3}$& $\varkappa^1_n$ & \centered{\pbox{20cm}{$\alpha_{n}\le[3;\overline{1,3}]; \ \ \alpha^*_{n-1}\le[0;3,\overline{3,1}]$}} & 1.123722 \\  [2ex] \hline
$\boldsymbol{3}13$  & $\varkappa^4_n$ & $([3;1,3,\overline{3,1}]+[0;1,\overline{1,3}])/4$  & 1.080930 \\ \hline
$\boldsymbol{3}1113$  & $\varkappa^4_n$ &  \centered{\pbox{20cm}{$([3;1,1,1,3,1,1,\overline{3,1,1,1}]+ [0;1,1,\overline{3,1,1,1}])/4 $}} & 1.050188 \\
\hline
$\boldsymbol{3}11111$  & $\varkappa^4_n$ &  \centered{\pbox{20cm}{$([3;1,1,1,1,1,1,\overline{3,1,1,1}]+ [0;1,1,\overline{3,1,1,1}])/4 $}} & 1.044287 \\
\hline
$111\boldsymbol{3}111$  & $\varkappa^4_n$ &  \centered{\pbox{20cm}{$([3;1,1,1,1,\overline{3,1,1,1}]+ [0;1,1,1,1,\overline{3,1,1,1}])/4 $}} & 1.054716 \\
\hline
\end{tabular}
\\
\\
The lemma is proven.
\end{proof}
We will call the pattern $(b_1,\ldots,b_k)$ \textit{prohibited} if the fact that this pattern occurs in the sequence $(a_1,a_2,\ldots)$ infinitely many times implies that  $\lambda^{[2]}(\alpha)>\lambda_{\infty}$. In other words, Lemma \ref{proh1} states that all the patterns from the first column of the table above are prohibited. The following corollary is an immediate consequence of Lemma \ref{proh1}. 
\begin{foll}
Suppose that $\lambda_2<\lambda^{[2]}(\alpha)\le\lambda_{\infty}$ for some irrational $\alpha=[a_0;a_1,\ldots,a_n,\ldots]$. Then there exists $N$ such that the sequence $(a_N,a_{N+1},\ldots)$ has the form $(31111(311)_{n_1}31111(311)_{n_2}31111\ldots)$, where all $n_i\ge 1$.
\end{foll}
Without loss of generality one can say that number $N$ from the previous corollary equals $1$. We will frequently refer to the following classical lemma concerning difference of two continued fractions.
\begin{lem}
\label{complem}
Let $\alpha=[a_0;a_1,\ldots,a_n,\alpha_{n+1}]$ and $\beta=[a_0;a_1,\ldots,a_n,\beta_{n+1}]$ be two continued fractions. Then
\begin{equation}
\label{cfracdiff}
\beta-\alpha=(-1)^{n+1}\frac{\beta_{n+1}-\alpha_{n+1}}{q_n^2(\alpha_{n+1}+\alpha^*_n)(\beta_{n+1}+\alpha^*_n)}, 
\end{equation}
where $q_n$ is the denominator of the convergent $\frac{p_n}{q_n}=[a_0;a_1,\ldots,a_n]$.
\end{lem}
\begin{proof}
Note that $\alpha^*_n=\beta^*_n$. Using Perron's formula one can easily see that
$$
\beta-\alpha=\biggl(\beta-\frac{p_n}{q_n}\biggr)-\biggl(\alpha-\frac{p_n}{q_n}\biggr)=\frac{(-1)^{n+1}}{q_n^2}\biggl(\frac{1}{\alpha_{n+1}+\alpha^*_n}-\frac{1}{\beta_{n+1}+\alpha^*_n}\biggr)=\frac{(-1)^{n+1}}{q_n^2}\frac{\beta_{n+1}-\alpha_{n+1}}{(\alpha_{n+1}+\alpha^*_n)(\beta_{n+1}+\alpha^*_n)}.
$$
\end{proof}
One can consider the value $q_t$ as the \textit{continuant} of the sequence $(a_1,\ldots,a_t)$. Let us give some definition. Suppose that $A$ is an arbitrary (possibly empty) finite sequence of positive integers. By $\langle A\rangle$ we denote its continuant. It is defined as follows: continuant of an empty sequence  $\langle\cdot \rangle$ equals $1$, $\langle a_1\rangle=a_1$, if $t\ge 2$ then one has
\begin{equation}
\label{contrule}
\langle a_1,a_2,\ldots,a_t\rangle=a_t\langle a_1,a_2,\ldots,a_{t-1}\rangle+\langle a_1,a_2,\ldots,a_{t-2}\rangle.
\end{equation}
One can see that
\begin{equation}
[a_0;a_1,\ldots,a_t]=\frac{\langle a_0, a_1, a_2,\ldots,a_t\rangle}{\langle a_1,a_2,\ldots,a_t\rangle}=\frac{p_t}{q_t}.
\end{equation}
Rule (\ref{contrule}) can be generalized as follows:
\begin{equation}
\begin{split}
\label{contrulegen}
\langle a_1,a_2,\ldots,a_t,a_{t+1},\ldots,a_s\rangle=\langle a_1,a_2,\ldots,a_{t}\rangle\langle a_{t+1},\ldots,a_{s}\rangle+
\langle a_1,a_2,\ldots,a_{t-1}\rangle\langle a_{t+2},\ldots,a_{s}\rangle=\\
=\langle a_1,a_2,\ldots,a_{t}\rangle\langle a_{t+1},a_{t+2},\ldots,a_{s}\rangle(1+[0;a_{t},a_{t-1},\ldots,a_1][0;a_{t+1},a_{t+2},\ldots,a_s]).
\end{split}
\end{equation}
Denote $\alpha_{\infty}:=[3;\overline{1,1,3}]$ and $\alpha^*_{\infty}:=[0;1,1,1,1,\overline{3,1,1}]$. Then $\lambda_{\infty}=\frac{\alpha_{\infty}+\alpha^*_{\infty}}{4}$.
\begin{lem}
\label{proh2m}
The pattern 
\begin{equation}
\label{pattern2m}
111(311)_{2m}3111
\end{equation}
is prohibited for all $m\ge 0$.
\end{lem}
\begin{proof} 
We will prove this statement by induction on $m$. The base case $m=0$ is already considered in Lemma \ref{proh1}. Suppose that for all $m\le k$ the pattern (\ref{pattern2m}) is prohibited. Let us also assume that the pattern (\ref{pattern2m}) for $m=k+1$ occurs in the sequence $(a_1,\ldots,a_n,\ldots)$ infinitely many times but $\lambda^{[2]}(\alpha)\le\lambda_{\infty}$. Let $n$ be an index of the first ''3'' in the group $(311)_{2k+2}$. Due to Lemma \ref{kappalem},  to get a contradiction it is enough to show that $\varkappa^4_n(\alpha)>\lambda_{\infty}$. This is equivalent to the inequality 
\begin{equation}
\label{diffsumequiv}
\alpha_n-\alpha_{\infty}>\alpha^{*}_{\infty}-\alpha^*_{n-1}.
\end{equation}
Let us obtain the lower estimate of $\alpha_n-\alpha_{\infty}$ using Lemma \ref{complem}. We know that
\begin{equation}
\label{alphanest}
\alpha_n-\alpha_{\infty}=[\underbrace{(3;1,1)_{2k+2},3,1,1}_{\text{coinciding part}},1,1,\ldots]-[\underbrace{(3;1,1)_{2k+2},3,1,1}_{\text{coinciding part}},3,1,1,\ldots].
\end{equation}
The first different partial quotient has index $6k+9$, hence by Lemma \ref{complem} one has
\begin{equation}
\begin{split}
\label{alphanest2}
\alpha_n-\alpha_{\infty}=\frac{[3;1,1,\ldots]-[1;1,1,\ldots]}{\langle 1,1,(3,1,1)_{2k+2}\rangle^2([3;1,1,\ldots]+[1;1,3,\ldots])([1;1,1,\ldots]+[1;1,3,\ldots])}>\\
\frac{[3;1,1]-[1;1]}{\langle 1,1,(3,1,1)_{2k+2}\rangle^2([3;1]+[1;1])([1;1]+[1;1])}=\frac{1}{16\langle1,1,(3,1,1)_{2k+2}\rangle^2}.
\end{split}
\end{equation}
Now we obtain the lower estimate of $\alpha^*_{n-1}=[0;1,1,1,1,3,\ldots]$. As the patterns $33$, $313$, and $1113111$ are prohibited, it extends as
$$
\alpha^*_{n-1}=[0;1,1,1,1,(3,1,1)_2,\ldots].
$$
As the we need the lower estimate of $\alpha^*_{n-1}$, we write
$$
\alpha^*_{n-1}\ge[0;1,1,1,1,(3,1,1)_2,3,\ldots].
$$
Which extends as
$$
\alpha^*_{n-1}\ge[0;1,1,1,1,(3,1,1)_2,3,1,1,\ldots].
$$
By the induction hypothesis, the pattern $111(311)_{2}3111$ is prohibited. Hence
$$
\alpha^*_{n-1}\ge[0;1,1,1,1,(3,1,1)_3,3,1,1\ldots]=[0;1,1,1,1,(3,1,1)_4\ldots].
$$
Repeating the same argument $k$ times, we obtain that
$$
\alpha^*_{n-1}\ge[0;1,1,1,1,(3,1,1)_{2k+2},3,1,1,1,1,\ldots].
$$
Let us now estimate the difference $\alpha^{*}_{\infty}-\alpha^*_{n-1}$ from above.
$$
\alpha^{*}_{\infty}-\alpha^*_{n-1}=[\underbrace{0;1,1,1,1,(3,1,1)_{2k+2},3,1,1}_{\text{coinciding part}},3,1,1,\ldots]-[\underbrace{0;1,1,1,1,(3,1,1)_{2k+2},3,1,1}_{\text{coinciding part}},1,1,\ldots].
$$
The first different partial quotient has index $6k+14$, hence by Lemma \ref{complem} one has
\begin{equation}
\begin{split}
\label{alphanest2star}
\alpha^{*}_{\infty}-\alpha^*_{n-1}=\frac{[3;1,1,\ldots]-[1;1,1,\ldots]}{\langle1,1,1,1,(3,1,1)_{2k+3}\rangle^2([3;1,1,\ldots]+[1;1,3,\ldots])([1;1,1,\ldots]+[1;1,3,\ldots])}<\\
\frac{[3;1]-[1;1,1]}{\langle1,1,1,1,(3,1,1)_{2k+3}\rangle^2([3;1,1]+[1;1,1])([1;1,1]+[1;1,1])}=\frac{1}{6\langle1,1,1,1,(3,1,1)_{2k+3}\rangle^2}.
\end{split}
\end{equation}
Comparing (\ref{alphanest2}) and (\ref{alphanest2star}), one can see that in order to obtain (\ref{diffsumequiv}) it suffices to show that
$$
\frac{\langle1,1,1,1,(3,1,1)_{2k+3}\rangle}{\langle 1,1,(3,1,1)_{2k+2}\rangle}>\sqrt{\frac{8}{3}}.
$$ 
The last inequality easily follows from the trivial estimates
$$
\langle1,1,1,1,(3,1,1)_{2k+3}\rangle>\langle1,1,1,1,3,1,1\rangle\langle(3,1,1)_{2k+2}\rangle=41\langle(3,1,1)_{2k+2}\rangle
$$
and
$$
\langle 1,1,(3,1,1)_{2k+2}\rangle<3\langle(3,1,1)_{2k+2}\rangle.
$$
Lemma is proven.
\end{proof}
\begin{lem}
\label{proh2mp1}
The pattern
\begin{equation}
\label{pattern2mpn}
31111(311)_{2m+1}31111(311)_{2k+1}31111,
\end{equation}
where $k>m\ge 0$ is prohibited.
\end{lem}
\begin{proof}
We will use the approach of Lemma \ref{proh2m}. Our goal is to show that the inequality
\begin{equation}
\label{diffsumequiv2}
\alpha_{\infty}-\alpha_n<\alpha^*_{n-1}-\alpha^{*}_{\infty}.
\end{equation}
holds for infinitely many $n$. Let $n$ be an index of the first ''3'' in the group $(311)_{2k+1}$. Once again, we obtain the upper estimate of $\alpha_{\infty}-\alpha_n$ using Lemma \ref{complem}.
\begin{equation}
\label{alphanest_min}
\alpha_{\infty}-\alpha_n=[\underbrace{(3;1,1)_{2k+1},3,1,1}_{\text{coinciding part}},3,1,1,3,1,1,\ldots]-[\underbrace{(3;1,1)_{2k+1},3,1,1}_{\text{coinciding part}},1,1,3,1,1,\ldots].
\end{equation}
The first different partial quotient has index $6k+6$. Using the estimates from (\ref{alphanest2star}), one can see that
\begin{equation}
\label{alphanest_minfin}
\alpha_{\infty}-\alpha_n<\frac{1}{6\langle 1,1,(3,1,1)_{2k+1}\rangle^2}.
\end{equation}
Now we need a lower estimate of $\alpha^*_{n-1}-\alpha^{*}_{\infty}$.
$$
\alpha^*_{n-1}-\alpha^{*}_{\infty}=[\underbrace{0;1,1,1,1,(3,1,1)_{2m+1},3,1,1}_{\text{coinciding part}},1,1,\ldots]-[\underbrace{0;1,1,1,1,(3,1,1)_{2m+1},3,1,1}_{\text{coinciding part}},3,1,1,\ldots].
$$
The first different partial quotient has index $6m+11$. Using the estimates from (\ref{alphanest2}), one can see that
\begin{equation}
\label{alphanest_minfin2}
\alpha^*_{n-1}-\alpha^{*}_{\infty}>\frac{1}{16\langle 1,1,1,1,(3,1,1)_{2m+2}\rangle^2}.
\end{equation}
 Comparing (\ref{alphanest_minfin}) and (\ref{alphanest_minfin2}), we deduce that in order to obtain (\ref{diffsumequiv2}) it suffices to show that
$$
\frac{\langle 1,1,(3,1,1)_{2k+1}\rangle}{\langle1,1,1,1,(3,1,1)_{2m+2}\rangle}>\sqrt{\frac{8}{3}}.
$$
As $k\ge m+1$, it is enough to verify the inequality
\begin{equation}
\label{alphanest_minfin3}
\frac{\langle 1,1,(3,1,1)_{2m+3}\rangle}{\langle1,1,1,1,(3,1,1)_{2m+2}\rangle}>\sqrt{\frac{8}{3}}
\end{equation}
for all $m\ge 0$. As 
$$
\langle 1,1,(3,1,1)_{2m+3}\rangle\ge\langle 1,1,(3,1,1)_{2m+2}\rangle\langle 3,1,1\rangle=7\langle 1,1,(3,1,1)_{2m+2}\rangle
$$
and
$$
\langle1,1,1,1,(3,1,1)_{2m+2}\rangle\le3\langle 1,1,(3,1,1)_{2m+2}\rangle,
$$
the inequality (\ref{alphanest_minfin3}) holds and the lemma is proven.
\end{proof}
As an immediate consequence of Lemmas \ref{proh2m} and \ref{proh2mp1}, we deduce the following corollary.
\begin{foll}
Suppose that $\lambda_2<\lambda^{[2]}(\alpha)<\lambda_{\infty}$ for some irrational $\alpha=[a_0;a_1,\ldots,a_n,\ldots]$. Then $\alpha\sim\xi_n$ for some $n\ge 3$.
\end{foll}
\section{Lower estimates}
In this section we show that $\lambda^{[2]}(\xi_i)=\lambda_i$ and thus complete the proof of Theorem \ref{thm1}. First, we prove two technical lemmas.
\begin{lem}
\label{3113113}
Let $\alpha=[a_0;a_1,\ldots]$ be an arbitrary irrational number. If $a_n$ is the middle ''3'' in the pattern $311\boldsymbol{3}113$, then
$$
\max(\varkappa^1_n(\alpha), \varkappa^2_n(\alpha), \varkappa^4_n(\alpha))<1.04.
$$ 
\end{lem}
\begin{proof}
 As the function $\varkappa^1_n(\alpha_n,\alpha^*_{n-1})$ is decreasing on both arguments and 
$$
\alpha_n\ge [3;1,1,\overline{3,1,1,1}],\quad \ \alpha^*_{n-1}\ge [0;1,1,\overline{3,1,1,1}],
$$
we see that
\begin{equation}
\label{est1}
\varkappa^1_n(\alpha)\le \frac{[3;1,1,\overline{3,1,1,1}]+[0;1,1,\overline{3,1,1,1}]}{(1+ [0;1,1,\overline{3,1,1,1}])([3;1,1,\overline{3,1,1,1}]-1)}\approx 1.031440.
\end{equation}
The quantities $\varkappa^2_n(\alpha)$ and $\varkappa^4_n(\alpha)$ are estimated in a similar way. As
$$
\alpha_{n+1}\le[1;1,\overline{3,1,1,1}],\quad \alpha^*_{n}\le [0;3,1,1,\overline{3,1,1,1}],
$$
we get the following estimate
\begin{equation}
\label{est2}
\varkappa^2_n(\alpha)\le \frac{[1;1,\overline{3,1,1,1}]+[0;3,1,1,\overline{3,1,1,1}]}{(1-[1;1,\overline{3,1,1,1}])(1+[0;3,1,1,\overline{3,1,1,1}])}\approx 1.031440.
\end{equation}
Finally, as
$$
\alpha_n\le [3;1,1,3,1,1,\overline{3,1,1,1}],\quad \ \alpha^*_{n-1}\le [0;1,1,3,1,1,\overline{3,1,1,1}]
$$
we see that
\begin{equation}
\label{est4}
\varkappa^4_n(\alpha)\le \frac{[3;1,1,3,1,1,\overline{3,1,1,1}]+[0;1,1,3,1,1,\overline{3,1,1,1}]}{4}\approx 1.030785.
\end{equation}
Combining the estimates (\ref{est1}), (\ref{est2}), and (\ref{est4}) yields the statement of the lemma.
\end{proof}
\begin{lem}
\label{311113113}
Let $\alpha=[a_0;a_1,\ldots]$ be an arbitrary irrational number. If $a_n$ is the middle ''3'' in the pattern $31111\boldsymbol{3}113$ or $31111\boldsymbol{3}113$, then
$$
\varkappa^4_n(\alpha)>\max(\varkappa^1_n(\alpha), \varkappa^2_n(\alpha)).
$$ 
\end{lem}
\begin{proof}
We will give a proof for the pattern $31111\boldsymbol{3}113$ only, as the proof for the second pattern is exactly same. One can easily see that $\varkappa^4_n(\alpha)>\varkappa^1_n(\alpha)$ if and only if the inequality
\begin{equation}
\label{ineqge4}
(1+\alpha^{*}_{n-1})(\alpha_n-1)>4
\end{equation}
holds. As
$$
 \alpha_n\ge [3;1,1,\overline{3,1,1,1}], \quad \alpha^{*}_{n-1}\ge [0;1,1,1,1,\overline{3,1,1,1}],
$$ 
we have the following estimate
$$
(1+\alpha^{*}_{n-1})(\alpha_n-1)\ge([1;1,1,1,1,\overline{3,1,1,1}])([2;1,1,\overline{3,1,1,1}])\approx 4.120747
$$
and the inequality (\ref{ineqge4}) holds.

Using the the obvious properties $\alpha_n=a_n+1/\alpha_{n+1}$ and $\alpha^*_n=1/(a_n+\alpha^*_{n-1})$ and the fact that $a_n=3$, one can transform $\varkappa^2_n$ and $\varkappa^4_n$ as follows:
$$
\varkappa^2_n(\alpha)=\frac{3\alpha_{n+1}+\alpha_{n+1}\alpha^{*}_{n-1}+1}{(2+\alpha^{*}_{n-1})(\alpha_{n+1} + 1)}, \quad
\varkappa^4_n(\alpha)=\frac{3\alpha_{n+1}+\alpha_{n+1}\alpha^{*}_{n-1}+1}{4\alpha_{n+1}}.
$$
Hence $\varkappa^4_n(\alpha)>\varkappa^2_n(\alpha)$ if and only if
\begin{equation}
\label{ineqge0}
\alpha_{n+1}\alpha^{*}_{n-1}+\alpha^{*}_{n-1}-2\alpha_{n+1}+2>0.
\end{equation}
We know that 
$$
\alpha_{n+1}\ge[1;1,\overline{3,1,1,1}], \quad \alpha^{*}_{n-1}\le[0;1,1,1,1,\overline{3,1,1,1}].
$$
Substituting these estimates to (\ref{ineqge0}), we can see that the inequality is satisfied and therefore $\varkappa^4_n(\alpha)>\varkappa^2_n(\alpha)$. Lemma is proven.
\end{proof}
\begin{foll}
$\lambda^{[2]}(\xi_i)=\lambda_i$.
\end{foll}
\begin{proof}
It follows directly from Lemmas \ref{3113113}  and \ref{311113113}.
\end{proof}
\begin{foll}
$\lambda_{\infty}\in\mathbb{L}_2$.
\end{foll}
\begin{proof}
Consider
\begin{equation}
\label{lambda0num}
\alpha=[0;(3,1,1)_{2n_1+1},3,1,1,1,1,(3,1,1)_{2n_2+1},3,1,1,1,1,(3,1,1)_{2n_3+1},\ldots],
\end{equation}
where $n_i$ is an arbitrary sequence of natural numbers that tends to infinity. By Lemmas \ref{3113113}  and \ref{311113113}, $\lambda^{[2]}(\alpha)=\lambda_{\infty}$. 
\end{proof}
Thus we completed the proof of Theorem \ref{thm1}. Also, from (\ref{lambda0num}) one can see that the set of real numbers satisfying $\lambda^{[2]}(\alpha)=\lambda_{\infty}$ has continuum many elements. Note that the set of numbers satisfying $\lambda^{[2]}(\alpha)=\lambda_{n}$ for each $n$ is countable as all such numbers are equivalent to $\xi_n$.

\textbf{Acknowledgements:} I would like to thank Nikolay Moshchevitin for giving me the formulation of the problem.

\noindent Dmitry Gayfulin,\\
Graz University of Technology, Institute of Analysis and Number Theory,\\
Steyrergasse 30/II, 8010 Graz, Austria\\
and\\ 
Big Data and Information Retrieval School \\
Faculty  of Computer Science    \\
National Research University Higher School of Economics \\
11 Pokrovsky boulevard, Moscow 109028 Russia\\
\textit{gamak.57.msk@gmail.com}

\begin{thebibliography}{3}

\bibitem{Cusick}
Cusick, T.W. and Flahive, M.E., \emph{The Markoff and Lagrange Spectra}, American Mathematical Society, 1989.

\bibitem{Divis}
Diviš В., Novak В., \emph{A remark on the theory of diophantine approximations}, Comment. Math. Univ. Carolinae, 12, No 1 (1971), 127—141.

\bibitem{khinchin1964}
Khinchin, A.Ya., \emph{Continued Fractions}, University of Chicago Press, 1964.

\bibitem{Lesca}
Lesсa J., \emph{Sur les approximations diophantiennes á une dimension}, L'Université de Grenoble (1968).

\bibitem{Morimoto}
Mоrimоtо S., \emph{Zur Theorie der Approximation einer irrationalen Zahl durch rationale Zahlen}, Tohoku Math. J., 45 (1938), 177—187.

\bibitem{Moshchevitin}
Moshchevitin N. G., \emph{Über die Funktionen des Irrationalitatsma\ss es}, Dubickas, A. (ed.) et al., Analytic and probabilistic methods in number theory. Proceedings of the sixth international conference, Palanga, Lithuania, September 11-17, 2016. Vilnius: Vilnius University Publishing House. 123-148 (2017).

\bibitem{Semenyuk}
Semenyuk P., \emph{On a problem related to "second" best approximations to a real number. Preprint available at: https://arxiv.org/abs/2303.12716}

\end{thebibliography}
\end{document}